\documentclass[reqno, 12pt]{amsart}
\usepackage{amsmath}
\usepackage{amsthm}
\usepackage{amsfonts}
\usepackage{amssymb}
\usepackage{tikz}
\usetikzlibrary{tikzmark,arrows.meta}
\usetikzlibrary{calc}
\usepackage{graphicx}
\usepackage{float}
\usepackage{hyperref}
\allowdisplaybreaks

\theoremstyle{plain}
\newtheorem{lemma}{Lemma}
\newtheorem{theorem}[lemma]{Theorem}

\newtheorem{corollary}[lemma]{Corollary}
\newtheorem{definition}[lemma]{Definition}

\title{Sums of Reciprocals of Generalized Triangular Numbers}

\author{Pawel Grzegrzolka}
\address{Syracuse University, Syracuse, USA}
\email{pgrzegrz@syr.edu}

\author{Jeffrey L. Meyer}
\address{Syracuse University, Syracuse, USA} 
\email{jlmeye01@syr.edu}

\date{\today} 

\keywords{triangular numbers, generalized triangular numbers, higher-order triangular numbers}
\subjclass[2020]{11B83, 40A05}

\begin{document}

\begin{abstract}
We compute the sum and the alternating sum of the reciprocals of triangular numbers using two standard methods from calculus: a telescoping series approach and a power series approach. We then extend these results to generalized (higher-order) triangular numbers and derive closed-form expressions for both the non-alternating and alternating series of all orders.
\end{abstract}

\maketitle

\section{Introduction}

Let $T_n = \sum_{i=1}^{n} i=\frac{n(n+1)}{2}$ be the $n^{\text{th}}$ triangular number.
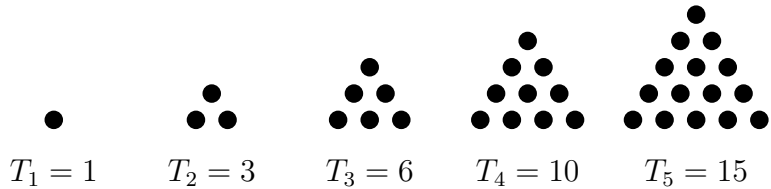
\begin{figure}[h]
\centering
\begin{tikzpicture}[scale=0.7]

\fill (0,0) circle (5pt);
\node at (0,-1) {$T_1 = 1$};

\fill (2.7,0) circle (5pt);
\fill (3.3,0) circle (5pt);
\fill (3,0.5) circle (5pt);
\node at (3,-1) {$T_2 = 3$};

\fill (5.4,0) circle (5pt);
\fill (6.0,0) circle (5pt);
\fill (6.6,0) circle (5pt);
\fill (5.7,0.5) circle (5pt);
\fill (6.3,0.5) circle (5pt);
\fill (6,1) circle (5pt);
\node at (6,-1) {$T_3 = 6$};

\fill (8.1,0) circle (5pt);
\fill (8.7,0) circle (5pt);
\fill (9.3,0) circle (5pt);
\fill (9.9,0) circle (5pt);
\fill (8.4,0.5) circle (5pt);
\fill (9.0,0.5) circle (5pt);
\fill (9.6,0.5) circle (5pt);
\fill (8.7,1.0) circle (5pt);
\fill (9.3,1.0) circle (5pt);
\fill (9.0,1.5) circle (5pt);
\node at (9,-1) {$T_4 = 10$};

\fill (11.0,0) circle (5pt);
\fill (11.6,0) circle (5pt);
\fill (12.2,0) circle (5pt);
\fill (12.8,0) circle (5pt);
\fill (13.4,0) circle (5pt);
\fill (11.3,0.5) circle (5pt);
\fill (11.9,0.5) circle (5pt);
\fill (12.5,0.5) circle (5pt);
\fill (13.1,0.5) circle (5pt);
\fill (11.6,1.0) circle (5pt);
\fill (12.2,1.0) circle (5pt);
\fill (12.8,1.0) circle (5pt);
\fill (11.9,1.5) circle (5pt);
\fill (12.5,1.5) circle (5pt);
\fill (12.2,2.0) circle (5pt);
\node at (12.2,-1) {$T_5 = 15$};
\end{tikzpicture}
\caption{The first five triangular numbers.}
\end{figure}
If you have taken calculus, you can likely deduce that the series of reciprocals of triangular numbers---taken either directly or with alternating signs---converges due to the quadratic denominator. What you may not know, but may wonder about, is to what values these sums converge.  That is, you may ponder the values of the sums
\vspace{-5pt}
\begin{equation}\label{easysums}
\sum_{n=1}^{\infty} \frac{1}{T_n} \quad\text{and}\quad \sum_{n=1}^{\infty} \frac{(-1)^{n+1}}{T_n}.
\end{equation}

\subsection*{Telescoping Series Approach}
Without realizing it, you may have already found the sum of the non-alternating series in your calculus class.  This problem is usually hidden as an exercise or example in a calculus textbook, typically in the section introducing series, and presented in the form $\sum_{n=1}^{\infty}\frac{1}{n(n+1)}$, without reference to the triangular numbers  (for example, see Example 6 in Section~8.2 of \cite{stewart}). A solution comes by way of the partial fraction decomposition and then the recognition of the series as a telescoping series.  We see that
\begin{align*}
\lim_{N\rightarrow\infty} \sum_{n=1}^{N}\frac{1}{n(n+1)}
=&\; \lim_{N\rightarrow\infty}\sum_{n=1}^{N}\left(\frac{1}{n} - \frac{1}{n+1}\right) \\
=&\;\lim_{N\rightarrow\infty} \left(1 - \frac{1}{2}+\frac{1}{2} - \frac{1}{3}+\frac{1}{3}- \frac{1}{4}+ \cdots +\frac{1}{N}-\frac{1}{N+1}\right)\\
=&\;\lim_{N\rightarrow\infty} \left(1-\frac{1}{N+1}\right) = 1.
\end{align*}
The sum of the reciprocals of triangular numbers is twice this value, namely
\begin{equation}\label{result12}
\sum_{n=1}^{\infty} \frac{1}{T_n}=\sum_{n=1}^{\infty} \frac{2}{n(n+1)}=2.
\end{equation}

An ambitious calculus student could apply a similar telescoping series approach to find the alternating sum of the reciprocals of triangular numbers. If we apply partial fraction decomposition to the alternating series, 
then, thanks to the telescoping nature of the terms and the accompanying alternating signs, nearly all the terms double instead of cancel, and we get that 

\begin{align*}
\sum_{n=1}^{N} \frac{(-1)^{n+1}}{n(n+1)} =&\; \sum_{n=1}^{N} (-1)^{n+1}\left(\frac{1}{n}-\frac{1}{n+1}\right)\\
=&\; \sum_{n=1}^{N} (-1)^{n+1}\frac{1}{n} +\sum_{n=2}^{N+1} (-1)^{n+1}\frac{1}{n}\\
=&\; 1 + 2\sum_{n=2}^{N}(-1)^{n+1}\frac{1}{n} + (-1)^{N+2}\frac{1}{N+1}\\
=&\; 1 + 2\left(\sum_{n=1}^{N}(-1)^{n+1}\frac{1}{n} -1\right) + (-1)^{N+2}\frac{1}{N+1}.
\end{align*}
The limit as $N\rightarrow \infty$ gives us that
\begin{equation}\label{result2}
\sum_{n=1}^{\infty} \frac{(-1)^{n+1}}{n(n+1)} = 1 + 2(\log 2 -1) = 2 \log 2 -1,
\end{equation}
where we used the well-known fact that $ \lim_{N\to\infty}\sum_{n=1}^{N}\frac{(-1)^{n+1}}{n}=\log 2$.
We multiply equation (\ref{result2}) by $2$ to get the alternating sum of the reciprocals of triangular numbers.  The sum is
\begin{equation}\label{result22}
\sum_{n=1}^{\infty} \frac{(-1)^{n+1}}{T_n} = 2\sum_{n=1}^{\infty} \frac{(-1)^{n+1}}{n(n+1)} = 4 \log 2 -2.
\end{equation}

\subsection*{Power Series Approach}

That someone would recognize a particular series as a telescoping series seems rather fortuitous.  Moreover, most calculus students would have not yet learned or remember that $ \lim_{N\to\infty}\sum_{n=1}^{N}\frac{(-1)^{n+1}}{n}=\log 2$. A slightly later calculus student could more directly approach these two sums by way of power series (Section 8.6 in \cite{stewart}). If we start with the fact that
\begin{equation*}
\frac{1}{1-x} = \sum_{n=0}^{\infty} x^n
\end{equation*}
for $-1<x<1$, and integrate this, we get that
\begin{equation*}
-\log(1-x) = C + \sum_{n=0}^{\infty}\frac{x^{n+1}}{n+1} = C + \sum_{n=1}^{\infty} \frac{x^n}{n}.
\end{equation*}
We set $x=0$, and determine that $C=0$; thus,
\begin{equation}\label{eqn1}
-\log(1-x) = \sum_{n=1}^{\infty} \frac{x^n}{n}.
\end{equation}
We note that convergence is now for $-1\leq x<1$, and, when establishing the convergence at $x =-1$, we get the previously mentioned (and lovely) fact that
\begin{equation}\label{nicefact}
\sum_{n=1}^{\infty} \frac{(-1)^{n+1}}{n} = \log 2.
\end{equation}
If we press on and integrate the relationship in (\ref{eqn1}), we get that
\begin{equation*}
(1-x)\log(1-x) - (1-x) = C + \sum_{n=1}^{\infty}\frac{1}{n}\cdot\frac{x^{n+1}}{n+1},
\end{equation*}
where we determine that $C=-1$ by setting $x=0$.  In other words, we have
\begin{equation}\label{eqn2}
(1-x)\log(1-x) +x = \sum_{n=1}^{\infty}\frac{x^{n+1}}{n(n+1)},
\end{equation}
with the interval of convergence of the right-hand sum and equality of the expressions for $-1\leq x \leq 1$, as long as the equality at $x=1$ is understood via the limit as $x\rightarrow~\!1^{-}$. We can evaluate (\ref{eqn2}) at $x=-1$ and, with the help of l'H\^{o}pital's rule, at $x=1$, to get
\begin{equation*}\label{result1}
\sum_{n=1}^{\infty}\frac{(-1)^{n+1}}{n(n+1)}=2\log 2 -1 \qquad \text{and} \qquad \sum_{n=1}^{\infty}\frac{1}{n(n+1)}=1.
\end{equation*}
We multiply these results by 2 to reach the desired sums of the reciprocals of triangular numbers,
\begin{equation*}
\sum_{n=1}^{\infty} \frac{(-1)^{n+1}}{T_n} = 4\log 2 -2 \quad \text{and} \quad \sum_{n=1}^{\infty} \frac{1}{T_n} = 2.
\end{equation*}
Notice that these values are the same (thank goodness) as the values obtained in (\ref{result12}) and (\ref{result22}) with the telescoping series approach.

\section{Sums of Reciprocals of Generalized Triangular Numbers}

With enough grit, we can extend the results of the previous section to generalized triangular numbers.  We define generalized triangular numbers (also known as higher-order triangular numbers) as follows. First, let $T_0(n) =1$ for all $n \in \mathbb{N}$.  Then, for any positive integer $k$ (where $k$ denotes the order), define $T_{k}(n) = \sum_{i=1}^{n} T_{k-1}(i)$.  This means $T_1(n) = n$ and  $T_2(n) = n(n+1)/2$, the usual triangular numbers. The figure below gives the first few generalized triangular numbers of the first few orders.

\[
\begin{array}{c|cccccc}
 & n=1 &2 &3 &4 &5 &6 \\\hline
k=0 & 1 & 1 & 1 & 1 & 1 & 1 \\
k=1 & 1 & 2 & 3 & 4 & 5 & 6 \\
k=2 &
\tikzmarknode{t1}{1} &
\tikzmarknode{t2}{3} &
\tikzmarknode{t3}{6} &
\tikzmarknode{t4}{10} &
\tikzmarknode{t5}{15} &
21 \\
k=3 & 1 & 4 & 10 & 20 & 35  & 56 \\
k=4 & 1 & 5 & 15 & 35 & 70 & 126 \\
\end{array}
\]

\begin{tikzpicture}[overlay,remember picture]
  \node[red] at ($(t1)!0.5!(t2)$) {$+$};
  \node[red] at ($(t2)!0.5!(t3)$) {$+$};
  \node[red] at ($(t3)!0.5!(t4)$) {$+$};
  \node[red] at ($(t4)!0.5!(t5)$) {$+$};

\draw[red,thick,-{Stealth[length=1.5mm,width=2mm]}]
  ([yshift=-0.7mm]t1.south) -- ([yshift=-0.7mm]t5.south)
  -- ++(0,-2mm);
\end{tikzpicture}

By computing individual cases for values of $k$, one can guess that
\begin{equation*}
T_k(n) = \binom{n+k-1}{k}= \frac{n(n+1)\cdots(n+k-1)}{k!}.
\end{equation*}
The above formula is proved by noting that $T_0(n) = 1 = \binom{n-1}{0}$, along with an application of the ``hockey stick identity'' \cite{hockey-stick} to get that
\begin{equation*}
T_{k+1}(n) = \sum_{i=0}^{n} T_k(i) = \sum_{i=0}^{n}\binom{i+k-1}{k} = \binom{n+k}{k+1}.
\end{equation*}

For the reminder of this paper, we consider this generalization of the sums in (\ref{easysums}), namely
\begin{equation*}
\sum_{n=1}^{\infty} \frac{1}{T_k(n)} \quad\text{and}\quad \sum_{n=1}^{\infty} (-1)^{n+1} \frac{1}{T_k(n)}.
\end{equation*}
For $k=0$, both series diverge. For $k=1$, the first series diverges while the second one converges to $\log 2$ by (\ref{nicefact}). For $k=2$, the results are given in the previous section. To tackle other values of $k$, we employ the two approaches from the previous section, the telescoping series approach and the power series approach.

\subsection*{Telescoping Series Approach}

To use the telescoping series approach, we first need to recognize the telescoping nature of the series. We see this from the following lemma, which you can readily prove by using the closed form of a generalized triangular number.

\begin{lemma}\label{partialfraction}
For any integer $k>1$,
\begin{equation*}
\frac{1}{T_k(n)} = \frac{k}{k-1}\left(\frac{1}{T_{k-1}(n)} - \frac{1}{T_{k-1}(n+1)}\right).
\end{equation*}
\end{lemma}

This lemma together with the telescoping series approach give us the sum of the reciprocals of the generalized triangular numbers.
\begin{theorem}\label{telescope2}
For any integer $k>1$,
\begin{equation*}
\sum_{n=1}^{\infty} \frac{1}{T_k(n)} = \frac{k}{k-1}.
\end{equation*}
\end{theorem}

\begin{proof}
From Lemma~\ref{partialfraction}, we have that
\begin{align*}
\lim_{N\rightarrow\infty} \sum_{n=1}^{N} \frac{1}{T_k(n)} =&\; \lim_{N\rightarrow\infty} \sum_{n=1}^{N}\frac{k}{k-1}\left(\frac{1}{T_{k-1}(n)} - \frac{1}{T_{k-1}(n+1)}\right)\\
=&\; \lim_{N\rightarrow\infty} \frac{k}{k-1}\left(1 - \frac{1}{T_{k-1}(N+1)}\right) = \frac{k}{k-1}. \qedhere
\end{align*}
\end{proof}

We can also apply the telescoping series approach to the alternating series, where, as before, most terms double instead of cancel.

\begin{theorem}\label{telescope3}
For any integer $k>1$,
\begin{equation*}
\sum_{n=1}^{\infty}\frac{(-1)^{n+1}}{T_k(n)} = \frac{k}{k-1}\left(2\sum_{n=1}^{\infty}\frac{(-1)^{n+1}}{T_{k-1}(n)} - 1 \right).
\end{equation*}
\end{theorem}
\begin{proof}
Again, from Lemma~\ref{partialfraction}, we have
\begin{align*}
\sum_{n=1}^{N} \frac{(-1)^{n+1}}{T_k(n)} =&\;  \sum_{n=1}^{N}(-1)^{n+1} \frac{k}{k-1}\left(\frac{1}{T_{k-1}(n)} - \frac{1}{T_{k-1}(n+1)}\right)\\
=&\; \frac{k}{k-1}\left(\sum_{n=1}^{N}\frac{(-1)^{n+1}}{T_{k-1}(n)} + \sum_{n=1}^{N}\frac{(-1)^{n+2}}{T_{k-1}(n+1)}\right)\\
=&\; \frac{k}{k-1}\left(\sum_{n=1}^{N}\frac{(-1)^{n+1}}{T_{k-1}(n)} + \sum_{n=2}^{N+1}\frac{(-1)^{n+1}}{T_{k-1}(n)}\right)\\
=&\; \frac{k}{k-1}\left(2\sum_{n=1}^{N}\frac{(-1)^{n+1}}{T_{k-1}(n)}+ \frac{(-1)^{N+2}}{T_{k-1}(N+1)}-1 \right).
\end{align*}
Upon computing the limit as $N\rightarrow \infty$, we get that
\begin{align*}
\sum_{n=1}^{\infty}\frac{(-1)^{n+1}}{T_k(n)} =&\; \frac{k}{k-1}\left(2\sum_{n=1}^{\infty}\frac{(-1)^{n+1}}{T_{k-1}(n)} - 1 \right). \qedhere
\end{align*}
\end{proof}

Note that this result is not as satisfying as the one for the usual triangular numbers, as it involves the sum of the generalized triangular numbers of one degree lower. The recursive nature of the result is not surprising (though it appears unrelated), given the recursive definition of $T_k(n)$.  However, its advantage is that one can easily compute the alternating sum of generalized triangular numbers for the first few values of $k$. This, in turn, can reveal a pattern leading to a more general, closed form of this sum. The reader is encouraged to find the pattern and the associated closed form before reading any further.

\begin{theorem}\label{telescope4}
For any integer $k>1$,
\begin{equation*}
\sum_{n=1}^{\infty}\frac{(-1)^{n+1}}{T_k(n)} = k2^{k-1}\log 2 - k\sum_{i=1}^{k-1}\frac{2^{k-1-i}}{i}.
\end{equation*}
\end{theorem}
\begin{proof}
We proceed by induction on $k$. For $k=2$, we get (\ref{result22}), an equation proved in two different ways in the first section of this paper. Now, assume the theorem is true for $k$. 
By Theorem \ref{telescope3}, we know that
\[\sum_{n=1}^{\infty}\frac{(-1)^{n+1}}{T_{k+1}(n)} = \frac{k+1}{k}\left(2\sum_{n=1}^{\infty}\frac{(-1)^{n+1}}{T_{k}(n)} - 1 \right).\]
We apply the inductive hypothesis to get the desired result.
\begin{align*}
\sum_{n=1}^{\infty}\frac{(-1)^{n+1}}{T_{k+1}(n)}=&\; \frac{k+1}{k}\Biggl(2\biggl(k2^{k-1}\log 2 - k\sum_{i=1}^{k-1}\frac{2^{k-1-i}}{i}\biggr) - 1 \Biggr)\\
=&\; \frac{k+1}{k} \left( k2^{k}\log 2 - k \sum_{i=1}^{k-1}\frac{2^{k-i}}{i}-1\right)\\
=&\; (k+1)2^{k}\log 2 - (k+1) \left(\sum_{i=1}^{k-1}\frac{2^{k-i}}{i}+\frac{1}{k}\right)\\
=&\; (k+1)2^{k}\log 2 - (k+1) \sum_{i=1}^{k}\frac{2^{k-i}}{i}. \qedhere
\end{align*}
\end{proof}
It is worth pointing out that the formula in Theorem \ref{telescope4} also holds for $k=1$ (as long as we interpret an empty sum as $0$) even though the formula in Theorem \ref{telescope3} is not defined for $k=1$.


With the telescoping series approach, we found the general forms for the sum and the alternating sum of the reciprocals of generalized triangular numbers for any integer $k>1$.  Combining this with our earlier results for $k=0$ and $k=1$ gives us the following table.

\[
\renewcommand{\arraystretch}{1.5}
\begin{array}{c|c@{\qquad \qquad}c}
 & \displaystyle \sum_{n=1}^{\infty} \frac{1}{T_k(n)}
 & \displaystyle \sum_{n=1}^{\infty}\frac{(-1)^{n+1}}{T_k(n)} \\
\hline
k=0 & \text{divergent} & \text{divergent} \\
k=1 & \text{divergent} & \log 2 \\
k>1 & \displaystyle \frac{k}{k-1}
    & \displaystyle k2^{k-1}\log 2
      - k\sum_{i=1}^{k-1}\frac{2^{k-1-i}}{i}
\end{array}
\]

\subsection*{Power Series Approach}

The telescoping series approach for $k>1$ was dependent on recognizing that both series exhibited telescoping behavior. In other words, the telescoping series approach was dependent on Lemma \ref{partialfraction}. Is it possible to obtain the same results without using the telescoping nature of both series? In other words, can we obtain the same answers by using the power series approach?

When we used the power series approach on the usual triangular numbers, we integrated $ \frac{1}{1-x} = \sum_{n=0}^{\infty} x^n$ twice. For the generalized triangular numbers of the $k^{\text{th}}$ order, we will need to integrate that equation $k$ times. Doing so leads to certain regular constants of integration, which we define now.

\begin{definition}\label{H_definition}
For any integer $j \geq 1$, define 
\[C_j =  \frac{(-1)^j}{(j-1)!}H_{j-1},\]
where $H_0 = 0$ and $H_\ell = \sum_{i=1}^{\ell}\frac{1}{i}$; that is, $H_\ell$ is the $\ell^{\text{th}}$ partial sum of the harmonic series.
\end{definition}


The coefficients $C_j$ satisfy a recursive formula given by the following lemma, which one can readily prove by noting that $\displaystyle H_{j-1}=H_{j-2}+\frac{1}{j-1}$ for any integer $j \geq 2$.

\begin{lemma}\label{lemma1}
For any integer $j\geq 2$, 
\[C_{j}=\dfrac{-C_{j-1}}{j-1}+\dfrac{(-1)^{j}}{(j-1)!(j-1)}.\]
\end{lemma}

We can now state the result of integrating $ \frac{1}{1-x} = \sum_{n=0}^{\infty} x^n$ repeatedly.

\begin{theorem}\label{main_theorem}  For any integer $k \geq 1$ and for $-1\leq x \leq 1$ (if $k=1$, convergence is for $-1\leq x <1$),
\vspace{-10pt}
\begin{multline}
 (-1)^k\frac{(1-x)^{k-1}}{(k-1)!}\log(1-x) - C_k (1-x)^{k-1} + \sum_{j=1}^{k}
C_j\frac{x^{k-j}}{(k-j)!}\\ = \sum_{n=1}^{\infty}\frac{x^{n+k-1}}{n(n+1)\cdots(n+k-1)},\label{maineqn}
\end{multline}
where the case of $x=1$ is understood via the limit as $x\rightarrow 1^{-}$.
\end{theorem}

\begin{proof}
When $k=1$, note that $C_1 = 0$ and we get the well-known relationship noted in (\ref{eqn1}).
Assume that (\ref{maineqn}) holds for $k$.  We integrate (by parts) to get that
\vspace{-10pt}
\begin{multline*}
(-1)^{k+1}\frac{(1-x)^k}{k(k-1)!}\log(1-x) +(-1)^{k}\frac{(1-x)^k}{(k-1)!k^2}+ \frac{C_k}{k} (1-x)^k\\ + \sum_{j=1}^{k}
C_j\frac{x^{k-j+1}}{(k-j+1)!} + C = \sum_{n=1}^{\infty}\frac{x^{n+k}}{n(n+1)\cdots(n+k-1)(n+k)},
\end{multline*}
where $C$ is the constant of integration. Factoring $(1-x)^k$ from the second and third term gives
\begin{multline}\label{mainthm1}
(-1)^{k+1}\frac{(1-x)^k}{k(k-1)!}\log(1-x) +\left(\frac{(-1)^{k}}{(k-1)!k^2}+\frac{C_k}{k}\right) (1-x)^k\\ + \sum_{j=1}^{k}
C_j\frac{x^{k-j+1}}{(k-j+1)!} + C = \sum_{n=1}^{\infty}\frac{x^{n+k}}{n(n+1)\cdots(n+k-1)(n+k)}.
\end{multline}
By Lemma \ref{lemma1}, the coefficient of $(1-x)^k$ equals $-C_{k+1}$. We use this fact and set $x= 0$ in (\ref{mainthm1}) to see that $C=C_{k+1}$. Thus, we have
\vspace{-10pt}
\begin{multline}\label{mainthm2}
(-1)^{k+1}\frac{(1-x)^k}{k(k-1)!}\log(1-x) -C_{k+1} (1-x)^k + \sum_{j=1}^{k}
C_j\frac{x^{k-j+1}}{(k-j+1)!} + C_{k+1}\\ = \sum_{n=1}^{\infty}\frac{x^{n+k}}{n(n+1)\cdots(n+k-1)(n+k)}.
\end{multline}
The last two terms on the left side of the equation (\ref{mainthm2}) can be rewritten to get
\vspace{-8pt}
\begin{equation*}
\sum_{j=1}^{k}C_j\frac{x^{k-j+1}}{(k-j+1)!} + C_{k+1} = \sum_{j=1}^{k+1}
C_j\frac{x^{k-j+1}}{(k-j+1)!} = \sum_{j=1}^{k+1}
C_j\frac{x^{k+1-j}}{(k+1-j)!}.
\end{equation*}
This all implies that we have
\vspace{-10pt}
\begin{multline*}
(-1)^{k+1}\frac{(1-x)^{k}}{k!}\log(1-x) - C_{k+1} (1-x)^{k} + \sum_{j=1}^{k+1}
C_j\frac{x^{k+1-j}}{(k+1-j)!}\\ = \sum_{n=1}^{\infty}\frac{x^{n+k}}{n(n+1)\cdots(n+k-1)(n+k)},
\end{multline*}
as desired and required.
\end{proof}

The next step in the power series approach is to evaluate equation (\ref{maineqn}) at $x=-1$ and, with the help of L'Hospital rule, at $x=1$ (this will also justify why the interval of convergence in Theorem \ref{main_theorem} is $-1\leq x \leq 1$). To get the proper evaluation, we must multiply these results by $k!$ (the same way we multiplied our results by $2$ when we dealt with the usual triangular numbers).

We evaluate at $x=1$, to give the non-alternating sum.

\begin{corollary}\label{nontelescope1} For any integer $k>1$,
\begin{equation*}
\sum_{n=1}^{\infty}\frac{1}{T_k(n)}=\sum_{j=1}^{k} C_j \frac{k!}{(k-j)!}.
\end{equation*}
\end{corollary}
\begin{proof}
We evaluate (\ref{maineqn}) at $x=1$, with the aid of l'H\^{o}pital's rule. Then, we multiply everything by $k!$. The result is immediate.
\end{proof}
What is not immediate is the fact that the right side of Corollary~\ref{nontelescope1} equals $ k/(k-1)$, the result obtained in Theorem \ref{telescope2} using the telescoping series approach. Thankfully, the equality is true, as the following theorem shows.
\begin{theorem}\label{agreement}
For any integer $k> 1$,
\[\displaystyle \sum_{j=1}^{k} C_j \frac{k!}{(k-j)!} = \frac{k}{k-1}.\]
\end{theorem}
\begin{proof}
We proceed by induction.  For manageability in the proof, we work to show that
\begin{equation}\label{agreement1}
\sum_{j=1}^{k}C_j\frac{(k-1)!}{(k-j)!}=\frac{1}{k-1},
\end{equation}
which is equivalent to the equation we want to prove. The base case of $k=2$ is immediate, so we assume that (\ref{agreement1}) holds for $k$. Let $L= \sum_{j=1}^{k+1} C_j\frac{k!}{(k+1-j)!}$, the left side of equation (\ref{agreement1}) with $k$ replaced by $k+1$. Then, we have
\begin{align*}
L =&\; \sum_{j=1}^{k} C_j\frac{k!}{(k+1-j)!} + C_{k+1}k!\\
=&\; \sum_{j=1}^{k}\left( C_j\frac{k!}{(k+1-j)!} - C_j(j-1)\frac{(k-1)!}{(k+1-j)!} + C_j(j-1)\frac{(k-1)!}{(k+1-j)!}\right)\\
& + C_{k+1}k!\\
=&\;  \sum_{j=1}^{k} C_j\frac{(k-1)!}{(k+1-j)!}\bigl(k-(j-1)\bigr) + \sum_{j=1}^{k} C_j(j-1)\frac{(k-1)!}{(k+1-j)!} + C_{k+1}k!\\
=&\;  \sum_{j=1}^{k} C_j\frac{(k-1)!}{(k-j)!} + \sum_{j=1}^{k} C_j(j-1)\frac{(k-1)!}{(k+1-j)!} + C_{k+1}k!\\
=&\; \frac{1}{k-1} + \sum_{j=1}^{k} C_j(j-1)\frac{(k-1)!}{(k+1-j)!} + C_{k+1}k!,
\end{align*}
where in the last step we applied the inductive hypothesis. We note that $C_1=0$ and apply Lemma \ref{lemma1} to get that
\begin{align*}
L =&\; \frac{1}{k-1} +\sum_{j=2}^{k}\left(-\frac{C_{j-1}}{j-1}+\frac{(-1)^j}{(j-1)(j-1)!}\right)(j-1)\frac{(k-1)!}{(k+1-j)!} \\
&+ \left(-\frac{C_k}{k}+\frac{(-1)^{k+1}}{k\cdot k!}\right)k!\\
=&\; \frac{1}{k-1} -\sum_{j=2}^{k} C_{j-1}\frac{(k-1)!}{(k+1-j)!} - C_k (k-1)!+ \sum_{j=2}^{k} \frac{(-1)^{j}(k-1)!}{(k+1-j)!(j-1)!}\\
&+\frac{(-1)^{k+1}(k-1)!}{k!}\\
=&\;  \frac{1}{k-1} - \sum_{j=1}^{k-1} C_{j}\frac{(k-1)!}{(k-j)!} - C_k (k-1)! +\sum_{j=2}^{k+1}\frac{(-1)^{j}(k-1)!}{(k+1-j)!(j-1)!}\\
=&\; \frac{1}{k-1} -  \sum_{j=1}^{k} C_{j}\frac{(k-1)!}{(k-j)!}  +\sum_{j=1}^{k}\frac{(-1)^{j+1}(k-1)!}{(k-j)!j!}.
\end{align*}
Now, we apply the inductive hypothesis (again), and get
\begin{equation*}
L=  \frac{1}{k-1} - \frac{1}{k-1} - \frac{1}{k}\sum_{j=1}^{k}(-1)^j\binom{k}{j} = -\frac{1}{k} \left(\sum_{j=0}^{k}(-1)^j\binom{k}{j} - 1\right) = \frac{1}{k},
\end{equation*}
where we use the well-known fact that $\sum_{j=0}^{k}(-1)^j\binom{k}{j} = 0$. This completes the proof.
\end{proof}

Now that Theorem~\ref{agreement} is proved, we have an alternative proof of Theorem~\ref{telescope2}, which uses the power series approach and does not depend on the recognition of the series as telescoping.

Next, we apply Theorem~\ref{main_theorem} with $x=-1$ to evaluate the alternating sum of the reciprocals of generalized triangular numbers.

\begin{corollary}\label{alt_sum_result}
For any integer $k>1$,
\begin{equation*}
\sum_{n=1}^{\infty}  \frac{(-1)^{n+1}}{T_k(n)}   =  k2^{k-1}\log 2 + (-1)^{k+1} k!C_k 2^{k-1} + \sum_{j=1}^{k} C_j (-1)^j\frac{k!}{(k-j)!}.
\end{equation*}
\end{corollary}

\begin{proof}
We set $x=-1$ in equation (\ref{maineqn}) and multiply the result by $(-1)^kk!$ to get the desired equality.
\end{proof}

What remains is to show that the right-hand side of Corollary \ref{alt_sum_result} equals
\[k2^{k-1}\log 2 - k\sum_{i=1}^{k-1}\frac{2^{k-1-i}}{i}.\]
This will complete the power-series derivation of Theorem~\ref{telescope4}.  To achieve this, we use a standard binomial-harmonic identity (see \cite[Eq. (39)]{paule_schneider} or \cite[Identity 14]{spivey}).  For completeness, we include an elementary proof along the lines of the methods in \cite{knuth}.

\begin{lemma}\label{harmonic_partial_sums}
For any integer $n \geq 0$,
\[\displaystyle \sum_{m=0}^{n} \binom{n}{m} H_m = 2^{n} H_n - \sum_{i=1}^{n} \frac{2^{n-i}}{i},\]
where $H_\ell$ is as defined in Definition~\ref{H_definition}.
\end{lemma}

\begin{proof}

The equality clearly holds for $n=0$ and $n=1$. Assume the identity holds for $n$. Define $S_n = \sum_{m=0}^{n} \binom{n}{m} H_m$.
 Consider that
\begin{align*}
S_{n+1} &=\sum_{m=0}^{n+1} \binom{n+1}{m} H_m 
=  \sum_{m=0}^{n+1} \left[ \binom{n}{m} + \binom{n}{m-1} \right] H_m \\
&= \sum_{m=0}^{n} \binom{n}{m} H_m + \sum_{m=1}^{n+1} \binom{n}{m-1} H_m = S_n + \sum_{k=0}^{n} \binom{n}{k} H_{k+1},
\end{align*}
where in the first step we use Pascal's identity. We proceed from above, noting that $H_{k+1} = H_k+1/(k+1)$, to get that

\begin{align*}
S_{n+1} &= S_n + \sum_{k=0}^{n} \binom{n}{k} \left( H_k + \frac{1}{k+1} \right) = S_n + S_n + \sum_{k=0}^{n} \binom{n}{k} \frac{1}{k+1} \\
&= 2S_n + \sum_{k=0}^{n} \binom{n}{k} \frac{n+1}{k+1}\frac{1}{n+1} = 2S_n + \sum_{k=0}^{n} \binom{n+1}{k+1} \frac{1}{n+1} \\
&= 2S_n + \frac{1}{n+1}\sum_{k=0}^{n} \binom{n+1}{k+1} = 2S_n + \frac{1}{n+1} (2^{n+1} - 1),
\end{align*}
where, in the fourth step, we used the absorption identity for binomial coefficients,
$\binom{n}{k} \frac{n+1}{k+1}=\frac{(n+1)!}{(k+1)!}$,
and, in the sixth step, we used the fact that the sum of binomial coefficients up to the $n+1^{\text{st}}$ term is $2^{n+1}$. 

Now, we apply the inductive hypothesis $ S_n = 2^n H_n - \sum_{i=1}^n \frac{2^{n-i}}{i}$ to get that
\begin{align*}
S_{n+1} &= 2 \left( 2^n H_n - \sum_{i=1}^n \frac{2^{n-i}}{i} \right) + \frac{2^{n+1}-1}{n+1} \\
&= 2^{n+1} H_n - \sum_{i=1}^n \frac{2^{n-i+1}}{i} + \frac{2^{n+1}}{n+1} - \frac{1}{n+1} \\
&= 2^{n+1} \left( H_n + \frac{1}{n+1} \right) - \left( \sum_{i=1}^n \frac{2^{n+1-i}}{i} + \frac{2^0}{n+1} \right) \\
&= 2^{n+1} H_{n+1} - \sum_{i=1}^{n+1} \frac{2^{n+1-i}}{i},
\end{align*}
which completes the inductive step.
\end{proof}

We prove the connecting result.

\begin{theorem}\label{the_most_difficult_theorem}
For any integer $k>1$,
\vspace{-10pt}
\begin{multline}\label{difficult_theorem_equation}
k2^{k-1}\log 2 + (-1)^{k+1} k!C_k 2^{k-1} + \sum_{j=1}^{k} C_j (-1)^j\frac{k!}{(k-j)!}\\= k2^{k-1}\log 2 - k\sum_{i=1}^{k-1}\frac{1}{i}2^{k-1-i}.
\end{multline}
\end{theorem}
\begin{proof}
From Definition~\ref{H_definition}, the second term on the left side of (\ref{difficult_theorem_equation}) can be written as
\begin{align}\label{equation_hard_theorem_1}
(-1)^{k+1} k! C_k 2^{k-1}
&= (-1)^{k+1} k!
\left( (-1)^k \frac{H_{k-1}}{(k-1)!} \right) 2^{k-1} \nonumber\\
&= -\,k \, 2^{k-1} H_{k-1}. 
\end{align}
In a similar vein, we rewrite the summation on the left side of (\ref{difficult_theorem_equation}).
\begin{align*}
\sum_{j=1}^{k} C_j (-1)^j \frac{k!}{(k-j)!}
&= \sum_{j=1}^{k} \frac{H_{j-1}}{(j-1)!} \frac{k!}{(k-j)!}.
\end{align*}
We reindex this sum to get that
\[
\sum_{j=1}^{k} \frac{H_{j-1}}{(j-1)!} \frac{k!}{(k-j)!} = \sum_{m=0}^{k-1} \frac{H_m}{m!} \frac{k!}{(k-1-m)!}.\]
Note that $\frac{k!}{m!(k-1-m)!}
= k \binom{k-1}{m}$.
Therefore,
\begin{equation}\label{equation_hard_theorem_2}
\sum_{j=1}^{k} C_j (-1)^j \frac{k!}{(k-j)!}
= k \sum_{m=0}^{k-1} \binom{k-1}{m} H_m.
\end{equation}
We use (\ref{equation_hard_theorem_1}) and (\ref{equation_hard_theorem_2}) on the left-hand side of the original equation to get
\begin{align}\label{rewriting}
&k2^{k-1}\log 2
- k2^{k-1} H_{k-1}
+ k \sum_{m=0}^{k-1} \binom{k-1}{m} H_m.
\end{align}
Now, we employ Lemma \ref{harmonic_partial_sums}
 with $n = k-1$ to the sum in (\ref{rewriting}) to see that the left-hand side is equal to
\begin{align*}
&k2^{k-1}\log 2
- k2^{k-1} H_{k-1}
+ k\left(2^{k-1} H_{k-1}
- \sum_{i=1}^{k-1} \frac{2^{k-1-i}}{i}\right).
\end{align*}
This results in
\[
k2^{k-1}\log 2
- k \sum_{i=1}^{k-1} \frac{2^{k-1-i}}{i},
\]
which is exactly the desired right-hand side of (\ref{difficult_theorem_equation}).
\end{proof}

Note that Corollary \ref{alt_sum_result} and Theorem \ref{the_most_difficult_theorem} provide a telescoping-free proof of Theorem~\ref{telescope4}.

\end{document}